\documentclass[12pt]{amsart}

\usepackage{txfonts}
\usepackage{bm}
\usepackage{amsfonts,amssymb}
\usepackage{dsfont}
\usepackage{amscd}
\usepackage{extarrows}
\usepackage{amsmath}
\usepackage{enumerate}
\usepackage{amscd}
\usepackage[all]{xy}
\usepackage[hyperfootnotes=true]{hyperref}
\usepackage{mathrsfs}

\theoremstyle{fancy}
\newtheorem{thm}{Theorem}[section]

\newtheorem{lemma}[thm]{Lemma}

\newtheorem{prop}[thm]{Proposition}
\newtheorem{definition}{Definition}[section]

\newtheorem{que}{Question}[section]

\newtheorem{rmk}{Remark}[section]

\begin{document}

\title[Asymptotic dimensions of cohomology groups]{Asymptotic estimate of
cohomology groups valued in pseudo-effective line bundles}
\author{Zhiwei Wang}

\address{ School
of Mathematical Sciences\\Laboratory of Mathematics and Complex Systems\\Beijing Normal University
\\ Beijing 100875\\ P. R. China}
\email{zhiwei@bnu.edu.cn}

\author{Xiangyu Zhou}
\address{Institute of Mathematics\\Academy of Mathematics and Systems Sciences,
and Hua Loo-Keng Key Laboratory of Mathematics
\\Chinese Academy of Sciences\\Beijing 100190\\P. R. China}
\email{xyzhou@math.ac.cn}

\begin{abstract}In this paper, we study questions of Demailly and Matsumura on  the
asymptotic behavior of dimensions of cohomology groups for high tensor powers of
(nef) pseudo-effective line bundles over non-necessarily projective algebraic manifolds. By generalizing Siu's 
$\partial\overline{\partial}$-formula and Berndtsson's eigenvalue
estimate of $\overline{\partial}$-Laplacian and combining
Bonavero's technique, we obtain the following result: given a
holomorphic pseudo-effective line bundle $(L, h_L)$ on a compact
Hermitian manifold $(X,\omega)$, if $h_L$ is a singular metric with
algebraic singularities, then $\dim H^{q}(X,L^k\otimes E\otimes
\mathcal{I}(h_L^{k}))\leq Ck^{n-q}$ for $k$ large, with $E$ an
arbitrary holomorphic vector bundle. As applications, we obtain 
partial solutions to the questions of Demailly and Matsumura.
\end{abstract}
\date{November 28, 2017}
\subjclass[2010]{32J25, 32J27, 14C17, 14C20} \keywords{Nef line
bundle, pseudo-effective line bundle, singular metric, multiplier
ideal sheaf}
\thanks{The first author was partially supported by the Fundamental Research Funds for the Central Universities and by the NSFC grant NSFC-11701031, the second author was partially supported by the National Natural Science Foundation of China.}

\maketitle


\section{Introduction}

Numerical properties of cohomology groups valued in bundles play an important role to
approach certain fundamental problems of complex algebraic geometry
and complex analytic geometry. The concept of positivity is often
involved in the study. If some "strong positivity"  is satisfied,
one can derive  precise or asymptotic vanishing theorems of the
cohomology groups,  which can be  used to study, say embedding
problems, asymptotics of linear systems, extension problems of
holomorphic sections, the minimal model program, for listing just a few
(cf. \cite{AN54,
AG62,Bouche95,Cao13,Dem15,En93,GZ15,Laz041,Laz042,Mat14,Mat15,Siu82,SS85,Zhou13}).\\

When only some "weaker positivity" can be assumed, some
precise or asymptotic estimate of the cohomology groups are also
expected, which is again used to study the algebraic and analytic geometric consequences about the manifolds. For instance, in this aspect, one has the  Grauert-Riemenschneider conjecture (G-R
conjecture for short) and the abundance conjecture. 

The G-R
conjecture says that given a hermitian holomorphic line bundle over a
compact Hermitian manifolds, if the curvature form of the line bundle
is semi-positive and positive on an open dense subset, then the base
manifold is Moishezon, i.e. birational to a projective
manifold. 

Siu \cite{Siu84} solved the G-R conjecture by giving an
asymptotic estimate of the Dolbeault cohomology group. Shortly
later, Demailly \cite{Dem85, Dem89} gave another solution to the G-R
conjecture by establishing the celebrated holomorphic Morse
inequalities (giving asymptotic bounds on the cohomology of  tensor
bundles of holomorphic line bundles), which is an important development of
 the Riemann-Roch formula, and is used to study the
Green-Griffiths-Lang conjecture by Demailly in \cite{Dem11} recently. Bonavero
\cite{Bon98} considered the singular case and founded  singular
holomorphic Morse inequalities for line bundles admitting a singular
metric with algebraic singularities, which was used to establish G-R
type criterions by volumes of pseudo-effective line bundles  by
Boucksom and  Popovici \cite{Bou02, Pop08}. Berndtsson \cite{Ber02}
obtained an asymptotic eigenvalue estimate of
$\overline{\partial}$-Laplace which also implies the G-R conjecture.

The abundance conjecture \cite{CHP16,HP17,LOP16,LP161,LP162,Siu11}
asserts that $\kappa(X)=\nu(X)$ which is still an open question in algebraic geometry, where 
$\kappa(X)$ is the Kodaira dimension of the canonical line
bundle $K_X$ on the projective manifold $X$ and $\nu(X)$ is the numerical dimension of
$K_X$.\\

Let us recall some positivity concepts for holomorphic line bundles. Let $(X,\omega)$ be
a compact Hermitian manifold of complex dimension $n$, $L\rightarrow X$ be a holomorphic line bundle over $X$.

\begin{itemize}
\item  $L$ is said to be semi-positive (positive), if there is a smooth Hermitian metric $h$ of $L$,
such that the curvature $i\Theta_h\geq 0$ ($> \varepsilon\omega$ for some $\varepsilon>0$).
\item  $L$ is said to be pseudo-effective (big), if there is a singular Hermitian metric $h$ of $L$,
such that the curvature current $i\Theta_h\geq 0$ ($> \varepsilon\omega$ for some $\varepsilon>0$)
in the sense of currents.
\item $L$ is said to be nef (numerically effective or numerically eventually free),
if for any $\varepsilon>0$, there is a smooth Hermitian metric $h$ of $L$, such that
the curvature $i\Theta_h\geq -\varepsilon\omega$.
\end{itemize}

An Hermitian metric $h$ of $L$ is said to be singular, if locally
we can write $h=e^{-2\varphi}$, with $\varphi\in L^1_{loc}$. The
multiplier ideal sheaf $\mathcal{I}(\varphi)$ is the ideal subsheaf
of the germs of holomorphic functions $f\in \mathcal{O}_{X,x}$ such
that $|f|^2e^{-2\varphi}$ is integrable with respect to the Lebesgue
measure in local coordinates near $x$. 

Let $h_L=e^{-2\varphi}$ be a
singular Hermitian metric on $L$ where $\varphi\in
L^1_{loc}(X,\mathbb{R})$.  The  multiplier ideal sheaf of $h_L$ is
defined by $\mathcal{I}(h_L)=\mathcal{I}(\varphi)$, which  is well-known to be coherent when $\varphi$ is locally a plurisubharmonic function up to a bounded function.\\

In this paper, we are going to study the following two questions by Demailly and Matsumura on the
asymptotic estimate of dimensions of cohomology groups valued in high tensor powers of
(nef) pseudo-effective line bundles over a compact Hermitian manifold which is not necessarily a projective algebraic manifold.

\begin{que}
[Demailly's question \cite{Dem10}]\label{Demailly question}For a holomorphic nef line bundle $L$
and a holomorphic vector bundle $E$ over a compact Hermitian $n$-fold $X$, does the following estimate holds:
\begin{align}\label{estimate 1}
\dim_{\mathbb{C}} H^{q}(X,L^k\otimes E)\sim O(k^{n-q})~~~~~~?
\end{align}
\end{que}
\begin{que}[Matsumura's question \cite{Mat14, Mat15-1}]\label{WZ que}
 \label{Mat vanish}Let $L$ be a line bundle on a compact Hermitian manifold  with a singular metric $h$
 whose curvature is (semi)-positive. Then, for any holomorphic vector bundle $M$ on $X$ and any $q\geq 0$, one asks if the following estimate holds:
\begin{align}\label{Matsumura est}
\dim_{\mathbb{C}} H^q(X,\mathscr{O}_X(M\otimes L^k)\otimes \mathcal{I}(h^k))\sim O(k^{n-q})~~~~~?
\end{align}
\end{que}

In Demailly's book \cite{Dem10}, the same estimate as in Question \ref{Demailly question} was proved
under the assumption that $X$ is projective algebraic. The proof relies on the projective algebraic condition of $X$,
i.e. the existence of an ample line bundle on $X$. However the existence of the ample line bundle is not guaranteed
for general compact Hermitian (even K\"{a}hler) manifold. Demailly wrote in his book that " Observe that the argument
does not work any more if $X$ is not algebraic. It seems to be unknown whether the $O(k^{n-q})$ bound
still holds in that case". We summarize the question as above  Question \ref{Demailly question}.
Note that the estimate of Question \ref{Demailly question} in the projective case was used to
study abundance conjecture (e.g. \cite{LOP16}).\\

In \cite{Mat14}, Matsumura gave a positive answer to   Question
\ref{WZ que} when $X$ is projective, and the existence of an  ample
line bundle on $X$ is essentially needed in the proof. This type of
estimate was used to prove Nadel type vanishing theorem via
injectivity theorems, and thus called asymptotic cohomology
vanishing theorems for high tensor powers of line bundles with
singular metrics. Furthermore, Matsumura \cite{Mat14} wrote that if
one can give a positive answer to Question \ref{WZ que} under the
condition that $X$ is a compact K\"{a}hler manifold, then the
corresponding vanishing theorems can be generalized to the
K\"{a}hler case. Question 1.2 for K\"{a}hler manifolds is also mentioned in Problem 3.4 in \cite{Mat15-1}. Here we summarize  Question \ref{WZ que} for Hermitian manifolds in its full generality.\\

Note that Berndtsson's  estimate answers both Question \ref{Demailly question} and Question \ref{WZ que}
under an extra assumption that $L$ is semi-positive, i.e. $L$ admits a smooth Hermitian metric
with semi-positive curvature.  

For the sake of convenience, we state
Berndtsson's result as follows: let $(X,\omega)$ be a compact Hermitian manifold with
Hermitian metric $\omega$, $E$ and $L$  holomorphic line bundles over $X$. Assume that $L$ is given
a metric of semipositive curvature. Take $q\geq 1$. Then if $0\leq \lambda\leq k$,
\begin{align*}
h^{n,q}_{\leq \lambda}(X, L^k\otimes E)\leq C(\lambda+1)^qk^{n-q}.
\end{align*}
If $1\leq k\leq \lambda$, then
\begin{align*}
h^{n,q}_{\leq \lambda}(X,L^k\otimes E)\leq C\lambda^n,
\end{align*}
where $h^{n,q}_{\leq \lambda}(X, L^k\otimes E)$ is the dimension of the linear span of
$\overline{\partial}$-closed $L^k\otimes E$-valued $(n,q)$-eigen-forms of the
$\overline{\partial}$-Laplacian $\Box$ with eigenvalue less than or equal to $\lambda$. In particular,
if $\lambda=0$, $h^{n,q}_{\leq 0}(X,L^k\otimes E)$ is just $\dim_{\mathbb{C}} H^q(X,K_X\otimes E\otimes L^k)$.
The proof of this result is a clever combination of  localization technique,
Siu's $\partial\overline{\partial}$-formula \cite{Siu82} and a result of Skoda \cite{Sko82}.\\

To apply Berndtsson's technique in our case, we find some difficulties which we can not overcome by direct use of the technique.
On one hand, one can not expect  a smooth semi-positive representative  in a nef  class, we need to
compensate the loss of arbitrarily small positivity. On the other hand, for a pseudo-effective line bundle,
the singularities of the singular metric can be very complicated.

Fortunately, by generalizing and combining the techniques of Siu, Berndtsson and Bonavero,  we can make some progress under the assumption that $L$ admits  a singular metric
with algebraic singularities.  Note that this type of assumption is often used to study
some important problems in algebraic geometry \cite{LP161,LP162}. \\

Now let us introduce the main results in this paper.

Firstly, we compute the  vector bundle version of the so called Siu's $\partial\overline{\partial}$-formula \cite{Siu82} in order 
to meet our needs. Given an $E\otimes L$-valued $(n,q)$-form $u$,
we define an associated $(n-q,n-q)$-form $T_u$. Running a similar procedure like in \cite{Ber02}, we get the following

\begin{prop}\label{ddbar formula 1}Let $(X,\omega)$ be a compact Hermitian manifold with Hermitian metric
$\omega$, $E$ and $L$ be holomorphic vector bundle of rank $r$ and holomorphic line bundle respectively.
Let $u$ be an $E\otimes L$-valued $(n,q)$-form. If $u$ is $\overline{\partial}$-close, the following inequality holds
\begin{align*}
i\partial\overline{\partial}(T_u\wedge\omega_{q-1})\geq(-2Re\langle\Box u,u\rangle+\langle\Theta_{E\otimes L}
\wedge\Lambda u,u\rangle-c|u|^2)\omega_n.
\end{align*}
The constant $c$ is equal to zero if $\overline{\partial}\omega_{q-1}=\overline{\partial}\omega_q=0$,
hence in particular if $\omega$ is K\"{a}hler.
\end{prop}

Actually, when $d\omega=0$, we can get the following identity for smooth $E\otimes L$-valued $(n,q)$-form $u$.
\begin{align}\label{bochner formula}
i\partial\overline{\partial}(T_u\wedge\omega_{q-1})=
(-2Re\langle\overline{\partial}\overline{\partial }^* u,u\rangle+\langle\Theta_{E\otimes L}
\wedge\Lambda u,u\rangle+|\overline{\partial}\gamma_u|^2+|\overline{\partial}^*u|^2-|\overline{\partial}u|^2)\omega_n.
\end{align}

It is worth to mention that, by  carefully checking the computations of the $\partial\overline{\partial}$-formula,
we are able to derive a similar formula in  the case that the metric of $L$ is singular.

\begin{prop}\label{singular case 1}Let $(X,\omega)$ be a compact K\"{a}hler manifold, $(E,h_E)\rightarrow X$
be a holomorphic Hermitian vector bundle over $X$, and $(L,h_L)$ be a holomrphic pseudo-effective line bundle
with singular metric $h_L$ such that $i\Theta_{L,h_L}\geq \gamma$ with $\gamma$ a continuous
real $(1,1)$-form on $X$. Suppose that $u$ is an $E\otimes L$-valued $(n,q)$-form. Then we have the following equality
\begin{align*}
i\partial\overline{\partial}(T_u\wedge\omega_{q-1})=(-2Re\langle\overline{\partial}
\overline{\partial }^* u,u\rangle+\langle\Theta_{E\otimes L}\wedge\Lambda u,u\rangle+|
\overline{\partial}\gamma_u|^2+|\overline{\partial}^*u|^2-|\overline{\partial}u|^2)\omega_n.
\end{align*}
Moreover, if $u$ is $\overline{\partial}$-closed, then we have
\begin{align*}
i\partial\overline{\partial}(T_u\wedge\omega_{q-1})\geq(-2Re\langle\Box u,u\rangle+\langle\Theta_{E\otimes L}
\wedge\Lambda u,u\rangle)\omega_n.
\end{align*}
\end{prop}

The  equality (\ref{bochner formula}) can be used to prove vanishing theorems which is equivalent to solve
$\overline{\partial}$-equations on compact K\"{a}hler manifolds. By using Proposition \ref{ddbar formula 1},
the main result of \cite{Ber02} is generalized to verctor bundle version.

\begin{thm}\label{vector bundle}
 Let $(X,\omega)$ be a compact Hermitian manifold with Hermitian metric $\omega$, $E$ be a  holomorphic
 vector bundle and $L$ be a   holomorphic line bundle over $X$. Assume that $L$ is given a metric
 of semipositive curvature. Take $q\geq 1$. Then if $0\leq \lambda\leq k$,
\begin{align*}
h^{n,q}_{\leq \lambda}(X, L^k\otimes E)\leq C(\lambda+1)^qk^{n-q}.
\end{align*}
If $1\leq k\leq \lambda$, then
\begin{align*}
h^{n,q}_{\leq \lambda}(X,L^k\otimes E)\leq C\lambda^n.
\end{align*}
\end{thm}

Then we consider the case of $L$ pseudo-effective with algebraic singularities, i.e. $L$ is equipped
a singular metric $h_L$ with algebraic singularities whose curvature current is semi-positive
in the sense of currents. 

By combining Bonavero's technique \cite{Bon98}, Berndtsson's technique \cite{Ber02},
and Theorem \ref{vector bundle}, we are able to obtain the following

\begin{thm}\label{main result}
Let $(X,\omega)$ be a compact Hermitian manifold of complex dimension $n$. $(L,h_L)$ is a pseudo-effective
line bundle over $X$, such that $h_L$ is a singular metric with algebraic singularities. $E$ is a
holomorphic vector bundle over $X$. Then we have the following estimate
\begin{align*}
\dim_{\mathbb{C}} H^{q}(X,L^k\otimes E\otimes \mathcal{I}(h_L^{k}))\leq Ck^{n-q}
\end{align*}
for $k$ large, where $\mathcal{I}(h_L^k)$ is the multiplier ideal sheaf associated to the metric $h^k_L$ of $L^k$.
\end{thm}

Thus we give a positive answer to Question \ref{WZ que} under the assumption of algebraic singularities.\\

Combining with an argument related to an exact sequence, it follows from Theorem \ref{main result} that the following
partial solution to Question \ref{Demailly question} holds:

\begin{thm}\label{application-partial answer 1}Let $X$ be a compact complex manifold, $E\rightarrow X$ be
a holomorphic vector bundle over $X$, and $L\rightarrow X$ be a holomorphic line bundle with a
singular Hermitian metric $h_L$ with algebraic singularities such that the curvature current of $h_L$ is
semi-positive. Assume that the dimension of the singular locus of $h_L$ is $m$. Then for $q>m$, we have that
\begin{align*}
\dim H^q(X,\mathscr{O}_X(E\otimes L^k))\leq Ck^{n-q}.
\end{align*}
\end{thm}

Furthermore, by a Diophantine approximation argument, we can weaken the assumption of algebraic singularities in Theorem \ref{application-partial answer 1}.

\begin{thm}\label{analytic singularity 1}
	Let $X$ be a compact complex manifold, $E\rightarrow X$ be
	a holomorphic vector bundle over $X$, and $L\rightarrow X$ be a holomorphic line bundle with a
	singular Hermitian metric $h_L$ with  analytic singularities such that the curvature current
	of $h_L$ is semi-positive. Let $h$ be an arbitrarily smooth Hermitian metric of $L$, and set $e^{-\psi}=h_L/h$. Suppose that there is a small $\varepsilon>0$, such that $he^{-(1+\delta)\psi}$  are singular metrics of $L$ with semi-positive curvature current for $|\delta|<\varepsilon$. Assume that the dimension of the singular locus of $h_L$ is $m$.
	Then for $q>m$, we have that
	\begin{align*}
	\dim H^q(X,\mathscr{O}_X(E\otimes L^k))\leq Ck^{n-q}.
	\end{align*}
\end{thm}

Combining with injectivity theorem obtained in \cite{Mat14}, we get the following two vanishing theorems.
\begin{thm}\label{vanishing1}Let $X$ be a compact K\"{a}hler manifold, and $L$ be a holomorphic
line bundle over $X$. Suppose that $L$ is pseudo-effective, and the singular metric $h_{\min}$  with
minimal singularities of $L$ is with algebraic singularities. Then we have that
\begin{align*}
H^q(X,\mathscr{O}_X(K_X\otimes L)\otimes\mathcal{I}(h_{\min}))=0  ~~~~~\mbox{~~~for~~~~}~~~q>n-\kappa(L).
\end{align*}
\end{thm}
\begin{thm}\label{vanishing2}
Let $X$ be a compact K\"{a}hler manifold, and $L$ be a holomorphic line bundle with non-negative
Kodaira-Iitaka dimension over $X$. Suppose that $L$ is pseudo-effective, and the Siu's metric
$h_{siu}$   of $L$ is with algebraic singularities. Then we have that
\begin{align*}
H^q(X,\mathscr{O}_X(K_X\otimes L)\otimes\mathcal{I}(h_{siu}))=0  ~~~~~\mbox{~~~for~~~~}~~~q>n-\kappa(L).
\end{align*}
\end{thm}

This paper is organized as follows: In Section 2, we recall some definitions and fundamental results which will be used.
In Section \ref{ddbar smooth}, we derive
$\partial\overline{\partial}$-formula for forms valued in vector
bundle on compact Hermitian manifolds including the case when the metric of the line
bundle is singular with local quasi-psh potential on compact
K\"{a}hler manifolds, which is a generalization of
Siu's $\partial\overline{\partial}$-formula on compact K\"{a}hler
manifolds, and therefore prove the Proposition \ref{ddbar formula 1}.
In Section \ref{Singular ddbar}, we  consider singular $\partial\overline{\partial}$-formula and
prove Proposition \ref{singular case 1}. In Section \ref{proof of vector bundle}, along the way
of Berndtsson, we generalize Berndtsson's eigenvalue
estimate of $\overline{\partial}$-Laplacian to the case of powers of
line bundle tensor with vector bundle and prove  Theorem \ref{vector bundle}. In Section \ref{proof}, we give the proof
of our main Theorem \ref{main result}. In Section \ref{application1}, we give a proof of
Theorem \ref{application-partial answer 1} which partially answers the question of Demailly. In Section \ref{application3}, we give the proof of two vanishing theorems, i.e.
Theorem \ref{vanishing1} and Theorem \ref{vanishing2}.

\section{Technical preliminaries }
\subsection{Algebraic singularities}
\begin{definition}\label{analytic singularity}
Let $L\rightarrow X$ be a holomorphic line bundle over a complex manifold $X$. Let $h=e^{-2\varphi}$ be
a singular metric of $L$. 

We say $h$ is a singular metric with analytic singularities if $\varphi$ is
a locally integrable function on $X$ which has locally the form
\begin{align}\label{AS}
\varphi=\frac{c}{2}\log(\sum_{j=1}^N|f_j|^2)+\psi,
\end{align}
 where $f_j$ are non-trivial holomorphic functions and $\psi$ is smooth and $c$ is a
 $\mathbb{R}_+$-valued, locally constant function on $X$. 
 
We call $h$ a singular metric
 with \textbf{algebraic singularities}, if  $c\in \mathbb{Q}_+$.
\end{definition}

\subsection{Multiplier ideal sheaves}\label{CAS}

\begin{definition}[cf. \cite{Dem10}]\label{BS} Let $L\rightarrow X$ be a holomorphic line bundle
over a complex manifold $X$. Let $h=e^{-2\varphi}$ be a singular metric of $L$ which has
analytic singularities, i.e., $\varphi$ locally has the form (\ref{AS}). Then
$\mathscr{I}=\mathscr{I}(\varphi/c)$  is defined to be the  ideal of germs of holomorphic functions
$h$ such that $|h|\leq Ce^{\varphi/c}$ for some constant $C$. i.e.
\begin{align*}
|h|\leq C(|f_1|+\cdots+|f_N|).
\end{align*}

This is a globally defined ideal sheaf on $X$, locally equal to the integral closure
$\overline{(f_1,\cdots,f_N)}$, and $\mathscr{I}$ is coherent on $X$. If $(g_1,\cdots,g_{N'})$ are
local generators of $\mathscr{I}$, we still have
\begin{align*}
\varphi=\frac{c}{2}\log(\sum_{j=1}^{N'}|g_j|^2)+\psi',
\end{align*}
where $\psi'$ is smooth.
\end{definition}

Usually, given a plurisubharmonic (psh for short) function $\varphi$, it is not easy to compute
the multiplier ideal sheaf $\mathcal{I}(\varphi)$. When $\varphi$ is with analytic singularities,
the following facts are collected from \cite{Dem10}.

\begin{itemize}
\item[(I)] If $\varphi$ has the form $\varphi=\sum\lambda_j\log|g_j|$ where $D_j=g^{-1}_j(0)$ are
nonsingular irreducible divisors with normal crossings. Then $\mathcal{I}(\varphi)$ is the
sheaf of functions $f$ on open sets $U\subset X$ such that
    \begin{align}
    \int_U|f|^2\prod|g_j|^{-2\lambda_j}dV<\infty.\notag
    \end{align}

Since locally the $g_i$ can be taken to the coordinate functions from a local coordinate system
$(z_1,\cdots, z_n)$, the condition is that $f$ is divisible by $\prod g_j^{m_j}$
where $m_j-\lambda_j>-1$ for each $j$, i.e. $m_j\geq \lfloor\lambda_j\rfloor$ (integral part). Hence
\begin{align}
\mathcal{I}(\varphi)=\mathcal{O}(-\lfloor D\rfloor)=\mathcal{O}(-\sum\lfloor\lambda_j\rfloor  D_j).\notag
\end{align}

\item[(II)]\label{MIS} For the general case of analytic singularities, suppose that
\begin{align}
\varphi\sim \frac{c}{2}\log(|f_1|^2+\cdots+|f_N|^2)\notag
\end{align}
near the poles. 

From Definition \ref{BS}, one can assume that the $(f_j)$'s are generators of
the integrally closed ideal sheaf $\mathscr{I}=\mathscr{I}(\varphi/c)$, defined as the
sheaf of holomorphic functions $f$ such that $|f|\leq C\exp(\varphi/c)$. 

There is a
smooth modification $\mu:\widetilde{X}\rightarrow X$ of $X$ such that $\mu^*\mathscr{I}$ is
an invertible sheaf $\mathcal{O}_{\widetilde{X}}(-D)$ associated with a normal crossing divisor
$D=\sum\lambda_jD_j$, where $(D_j)$ are the components of the exceptional divisor of
$\widetilde{X}$, and $\lambda_j\in \mathbb{N}_+$. 

Thus locally we have
\begin{align}
\varphi\circ\mu\sim c\sum\lambda_j\log|g_j|\notag
\end{align}
where $g_j$ are local generators of $\mathcal{O}(-D_j)$. So
\begin{align*}\mathcal{I}(\varphi\circ\mu)=  \mathcal{O}(-\sum\lfloor c\lambda_j\rfloor D_j),\end{align*}
and
\begin{align*}\mathcal{I}(\varphi)=\mu_*\mathcal{O}_{\widetilde{X}}
(\sum(\rho_j-\lfloor c\lambda_j\rfloor)D_j),
\end{align*} where
$R=\sum\rho_jD_j$ is the zero divisor of the Jacobian function $J_\mu$ of the modification map.
\end{itemize}

\subsection{Skoda's Lemma}
\begin{definition}
For a psh function $\varphi$ on an open set $\Omega\subset \mathbb{C}^n$, the Lelong number
of $\varphi$ at $x$ is defined to be
\begin{align*}
\nu(\varphi,x):=\liminf_{z\rightarrow x}\frac{\varphi(z)}{\log|z-x|}.
\end{align*}
\end{definition}

\begin{lemma}[\cite{Sko72}]\label{skoda}
Let $\varphi$ be a psh function on an open set $\Omega\subset \mathbb{C}^n$ and let $x\in \Omega$.
\begin{itemize}
\item[(i)] If $\nu(\varphi,x)<1$, then $e^{-2\varphi}$ is integrable in a neighborhood od $x$,
in particular, $\mathcal{I}(\varphi)_x=\mathcal{O}_{\Omega,x}$.
\item[(ii)] If $\nu(\varphi,x)\geq n+s$ for some integer $s\geq 0$, then $e^{-2\varphi}\geq C|z-x|^{-2n-2s}$
in a neighborhood of $x$ and $\mathcal{I}(\varphi)_x\subset\mathfrak{m}^{s+1}_{\Omega,x}$,
where $\mathfrak{m}^{s+1}_{\Omega,x}$ is the maximal ideal of $\mathcal{O}_{\Omega,x}$.
\end{itemize}
\end{lemma}

\subsection{An isomorphism theorem of cohomology groups}

\begin{lemma}[cf. \cite{MM07}]\label{canonical sheaf-modification}
Let $\pi:\widetilde{X}\rightarrow X$ be the blow-up of $X$ with smooth center $Y\subset X$.
$L\rightarrow X$ be a holomorphic line bundle with singular metric $h_L$. Assume that in the
neighborhood of any point of the exceptional divisor $D$ of $\pi$, a local weight $\varphi$ of the metric
$h_L$ ($h_L=e^{-2\varphi}$) satisfies
\begin{align*}
\varphi\circ\pi=c\log|f|+\psi,
\end{align*}
for some $c>0$, $f$ is a local definition function of $D$ and $\psi$ is quasi-psh. Then for any
$p>1/c$ and $q\geq 0$, we have
\begin{align*}
H^q(\widetilde{X}, K_{\widetilde{X}}\otimes \pi^*(L^p\otimes E)\otimes
\mathcal{I}(\pi^*h_{L^p}))\simeq H^q(X,K_X\otimes L^p\otimes E\otimes \mathcal{I}(h_{L^p})),
\end{align*}
i.e.
\begin{align*}
H^{n,q}(\widetilde{X}, \pi^*(L^p\otimes E)\otimes \mathcal{I}(\pi^*h_{L^p}))
\simeq H^{n,q}(X, L^p\otimes E\otimes \mathcal{I}(h_{L^p})),
\end{align*}
\end{lemma}
\begin{rmk}
A real function is said to be quasi-psh, if it can be written as the  sum of psh function and a
smooth function locally. Lemma \ref{canonical sheaf-modification} is a consequence of Leray spectral theorem.
\end{rmk}

\subsection{Lebesgue decomposition of a current}\label{Lebesgue decomposition}

For a measure $\mu$ on a manifold $M$ we denote by $\mu_{ac}$ and $\mu_{sing}$ the uniquely
determined absolute continuous and singular measures (with respect to the Lebesgue measure on $M$) such that
\begin{align*}
\mu=\mu_{ac}+\mu_{sing}
\end{align*}
which is called the Lebesgue decomposition of the measure $\mu$. 

If $T$ is a $(1,1)$-current of order $0$ on $X$,
written locally $T=i\sum T_{ij}dz_i\wedge d\overline{z}_j$, we defines its absolute continuous and
singular components by
\begin{align*}
T_{ac}&=i\sum (T_{ij})_{ac}dz_i\wedge d\overline{z}_j,\\
T_{sing}&=i\sum (T_{ij})_{sing}dz_i\wedge d\overline{z}_j.
\end{align*}

The Lebesgue decomposition of $T$ is then
\begin{align*}
T=T_{ac}+T_{sing}.
\end{align*}

Note that a positive $(1,1)$-current $T\geq 0$ is of order $0$. If $T\geq 0$, it follows that $T_{ac}\geq 0$ and $T_{sing}\geq 0$.
Moreover, if $T\geq \alpha$ for a continuous $(1,1)$-form $\alpha$,
then $T_{ac}\geq \alpha$, $T_{sing}\geq 0$. 

It follows from the
Radon-Nikodym theorem that $T_{ac}$ is (the current associated to) a
$(1,1)$-form with $L^1_{loc}$ coefficients. The form $T_{ac}(x)^n$
exists for almost all $x\in X$ and is denoted by $T^n_{ac}$.

Note that $T_{ac}$ in general is not closed, even when $T$ is, so that the decomposition doesn't induce
a significant decomposition at the cohomological level. 

However, when $T$ is a closed positive $(1,1)$-current
with analytic singularities along a subscheme $V$, the residual part $R$ in Siu decomposition
(cf. \cite{Siu74}) of $T$ is nothing but $T_{ac}$, and $\sum_k\nu(T,Y_k)[Y_k]$ is $T_{sing}$.

\begin{rmk}\label{local decomposition}
Suppose now $X$ is a compact complex manifold of complex dimension $n$. $(L,h_L)$ be a pseudo-effective
line bundle over $X$, such that $h_L$ is a singular metric with analytic singularities. Then there is
a smooth modification $\mu:\widetilde{X}\rightarrow X$, such that $\mu^*\mathscr{I}$ is an
invertible sheaf $\mathcal{O}_{\widetilde{X}}(-D)$ associated with  a normal crossing divisor
$D=\sum\lambda_jD_j$, where $(D_j)$ are the components of the exceptional divisor of $\widetilde{X}$.
Now locally we can write
\begin{align*}
\mu^*\varphi=\varphi\circ\mu=c\log|s_{D}|+\widetilde{\psi},
\end{align*}
where $s_D$ is the canonical section of $\mathcal{O}_{\widetilde{X}}(-D)$,   and $\widetilde{\psi}$ is a
smooth potential. This implies that we have the following  Lebesgue decomposition
\begin{align}\label{MC}
\frac{i}{\pi}\mu^*\Theta_{h_L}=\frac{i}{\pi}\partial\overline{\partial}(\mu^*\varphi)=c[D]+\beta
\end{align}
where $[D]$ is the current of integration over $D$ and $\beta$ is a smooth closed $(1,1)$-form.  From the
pseudo-effectiveness, i.e.  $\frac{i}{\pi}\mu^*\Theta_{h_L}\geq 0$, we can conclude that $\beta\geq 0$.
\end{rmk}

\subsection{Regularity lemma}
\begin{lemma}[{\cite[Lemma 9.3]{ZGZ12}}]\label{regularity lemma}
Let $\varphi$ be a function in $L^1_{loc}(\Omega)$ such that $\Delta \varphi$ is a measure, where $\Omega$ is
a domain in $\mathbb{R}^m$ $(m\geq 2)$ and $\Delta=\sum_{j=1}^m\frac{\partial^2}{({\partial x^j})^2}$,
then $\frac{\partial}{\partial x^j}\varphi$ is a function in $L^1_{loc}(\Omega)$.
\end{lemma}

\begin{rmk} If $\varphi$ is  a quasi-psh function, then $\varphi$ is in $L^1_{loc}$ and by
Lemma \ref{regularity lemma} $\nabla\varphi$ is also in  $L^1_{loc}$.
\end{rmk}

\section{$\partial\overline{\partial}$-formula for vector bundle valued forms}\label{ddbar smooth}

Let $(E,h_{\alpha\overline{\beta}})$ be a Hermitian holomorphic vector bundle of rank $r$  over an
$n$-dimensional compact Hermitian manifold $(X, \omega)
$. Let $(L, h_L=e^{-\psi}:=e^{-2\varphi})$ be a Hermitian holomorphic line bundle over $X$. 

The curvature form
\begin{align*}
\Theta_{\alpha\overline{\beta}}=-\sqrt{-1}\sum_{i,j}\Omega_{\alpha\overline{\beta}i\overline{j}}dz^i
\wedge dz^{\overline{j}}
\end{align*}
of $E$ is given by
\begin{align*}
\Omega_{\alpha\overline{\beta}i\overline{j}}=\partial_i\partial_{\overline{j}}
h_{\alpha\overline{\beta}}-h^{\lambda\overline{\mu}}\partial_ih_{\alpha\overline{\mu}}
\partial_{\overline{j}}h_{\lambda\overline{\beta}}.
\end{align*}

Let $D'$ be the $(1,0)$-part of the  Chern connection associated to the vector bundle $E\otimes L$ with respect
to the Hermitian metrics of $E$ and $L$. Then one have the following formula
\begin{align}\label{curvature}
D'\overline{\partial}+\overline{\partial}D'=\Theta_{E\otimes L},
\end{align}
and
\begin{align}\label{debar star}
\overline{\partial}^*=-*D'*.
\end{align}
For the detailed computations, we refer to \cite{Kod}.

Let
\begin{align*}
u^\alpha=\frac{1}{p!q!}\sum u^\alpha_{I_p\overline{J}^q}dz^{I_p}\wedge dz^{\overline{J}^q}
\end{align*}
be an $E$-valued $(p,q)$-form on $X$.

Let $u$ be an $E\otimes L$-valued $(n,q)$-form,  we define an associated $(n-q,n-q)$-form which
 in a local trivialization is written as
 \begin{align}\label{Tu}
 T_u=c_{n-q}h_{\alpha\overline{\beta}}\gamma^\alpha\wedge\overline{\gamma^\beta}e^{-\psi}
 \end{align}
 where $\gamma^\alpha=*u^\alpha$, $c_{n-q}=i^{(n-q)^2}$. 
 
 Here $*$ denote the Hodge operator of
 the Hermitian manifold $X$, defined by the formula
 \begin{align*}
 \xi\wedge \overline{*\xi}=|\xi|^2\omega_n
 \end{align*}
 where $\xi $ is a $(p,q)$-form on $X$. 
 
The relation $*u^\alpha=\gamma^\alpha$ can be expressed as
 \begin{align*}
u^\alpha=c_{n-q}\gamma^\alpha\wedge \omega_q,
 \end{align*}
and moreover we have
\begin{align*}
*\gamma^\alpha=(-1)^{n-q}c_{n-q}\gamma^\alpha\wedge \omega_q.
\end{align*}

In the following  we shall also use the relations $ic_q=(-1)^qc_{q-1}$ and $c_{q-1}=c_{q+1}$.

By direct computation, we get that
\begin{align*}
\overline{\partial}(h_{\alpha\overline{\beta}}\gamma^\alpha\wedge\overline{\gamma^\beta}e^{-\psi})=
h_{\alpha\overline{\beta}}(\overline{\partial}\gamma^\alpha\wedge
\overline{\gamma^\beta}+(-1)^{n-q}\gamma^\alpha\wedge\overline{D'\gamma^\beta})e^{-\psi}
\end{align*}
And then
\begin{align*}
&\partial\overline{\partial}(h_{\alpha\overline{\beta}}\gamma^\alpha\wedge\overline{\gamma^\beta}e^{-\psi})=
h_{\alpha\overline{\beta}}(D'\overline{\partial}\gamma^\alpha\wedge\overline{\gamma^\beta}+(-1)^{n-q+1}
\overline{\partial}\gamma^\alpha\wedge\overline{\overline{\partial}\gamma^\beta}+\\
&+(-1)^{n-q}D'\gamma^\alpha\wedge\overline{D'\gamma^\beta}+\gamma^\alpha\wedge
\overline{\overline{\partial}D'\gamma^\beta})e^{-\psi}.
\end{align*}

By using the commutator formula (\ref{curvature}), we can get that

\begin{align}\label{ddbar term}
&\partial\overline{\partial}(h_{\alpha\overline{\beta}}\gamma^\alpha\wedge\overline{\gamma^\beta}e^{-\psi})=
h_{\alpha\overline{\beta}}(\Theta_{E\otimes L}\wedge\gamma^\alpha\wedge
\overline{\gamma^\beta}-\overline{\partial}D'\gamma^\alpha\wedge\overline{\gamma^\beta}+\\
&+(-1)^{n-q+1}\overline{\partial}\gamma^\alpha\wedge\overline{\overline{\partial}\gamma^\beta}+
(-1)^{n-q}D'\gamma^\alpha\wedge\overline{D'\gamma^\beta}+\gamma^\alpha\wedge
\overline{\overline{\partial}D'\gamma^\beta})e^{-\psi}.
\notag\end{align}

By the formula (\ref{debar star}), we have that
\begin{align*}
\overline{\partial}^*u^\alpha=-*D'\gamma^\alpha=(-1)^{n-q}c_{n-q-1}D'\gamma^\alpha\wedge\omega_{q-1},
\end{align*}
which implies that
\begin{align}\label{dbardbarstar}
\overline{\partial}\overline{\partial}^* u^\alpha&=
(-1)^{n-q}c_{n-q-1}(\overline{\partial}D'\gamma^\alpha\wedge \omega_{q-1}+
(-1)^{n-q-1}D'\gamma^\alpha\wedge\overline{\partial}\omega_{q-1})\\
&=(-1)^{n-q}c_{n-q-1}(\overline{\partial}D'\gamma^\alpha\wedge \omega_{q-1}+
O(|\overline{\partial}^*u^\alpha||\overline{\partial}\omega_{q-1}|)).\notag
\end{align}

In complex geometry, $\partial\omega$ is called the torsion form of $\omega$ (the operator
$\tau:=[\Lambda,\partial\omega]$ is called the torsion operator), one can see that
\begin{align*}
\overline{\partial}\omega_{q-1}=\overline{\partial}\omega\wedge\omega_{q-2},
\end{align*}
where the torsion comes into the game. 

If the Hermitian metric $\omega$ is K\"{a}hler (i.e. $d\omega=0$),
this term disappears.

Multiply (\ref{ddbar term}) by $ic_{n-q}\omega_{n-q}$, we have five terms. By (\ref{dbardbarstar}),
the second term in (\ref{ddbar term}) equals
\begin{align*}
-h_{\alpha\overline{\beta}}\overline{\partial}D'\gamma^\alpha\wedge\overline{\gamma^\beta}e^{-\psi}=
-h_{\alpha\overline{\beta}}\overline{\partial}\overline{\partial}^* u^\alpha\wedge
\overline{\gamma^\beta}e^{-\psi}=-\langle \overline{\partial}\overline{\partial}^* u,u\rangle\omega_n
\end{align*}
up to an error of size $O(|\overline{\partial}^*u||\overline{\partial}\omega_{q-1}||u|)$. 

Since
the entire expression is real, the fifth term must be the conjugate of the second one, so these terms together give
\begin{align*}
-2Re\langle \overline{\partial}\overline{\partial}^* u,u\rangle\omega_n.
\end{align*}

The first term is the curvature term
\begin{align*}
\langle \Theta_{E\otimes L}\wedge \Lambda u,u\rangle\omega_n=
\langle (\Theta_{E}+\Theta_L\otimes\mathds{1}_E)\wedge \Lambda u,u\rangle\omega_n.
\end{align*}

Checking signs, we see that the forth term equals
\begin{align*}
|\overline{\partial}^*u|^2\omega_n,
\end{align*}
so it only need to analyse the third term. 

Consider the bilinear form on $E\otimes L$ valued $(n-q,1)$-forms defined by
\begin{align}\label{bilinear form 1}
[\chi,\eta]\omega_n=ic_{n-q}(-1)^{n-q+1}h_{\alpha\overline{\beta}}\chi^\alpha
\wedge\overline{\eta^{\beta}}\wedge \omega_{q-1}=
-c_{n-q+1}h_{\alpha\overline{\beta}}\chi^\alpha\wedge\overline{\eta^{\beta}}\wedge \omega_{q-1}.
\end{align}

Fix an arbitrary point $x\in X$, we can choose a good coordinate chart $(U,z)$ centered at $x$ and a
good trivialization of $E$ and $L$ such that $\omega=\frac{i}{2}\partial\overline{\partial}|z|^2$,
$h_{\alpha\overline{\beta}}=\delta_{\alpha\overline{\beta}}$ and $\psi=0$ at $x$. Then the bilinear form
(\ref{bilinear form 1}) at $x$ reads
\begin{align*}
[\chi,\eta]\omega_n=-c_{n-q+1}\sum_{\alpha=1}^r\chi^\alpha\wedge\overline{\eta^{\alpha}}\wedge \omega_{q-1}.
\end{align*}

For each $\alpha\in \{1,2, \cdots, r\}$, we consider
\begin{align}\label{bilinear form 2}
[\chi,\eta]\omega_n=-c_{n-q+1}\chi^\alpha\wedge\overline{\eta^{\alpha}}\wedge \omega_{q-1}.
\end{align}

It is proved in \cite{Ber02} that at $x$
\begin{itemize}
\item the form (\ref{bilinear form 2}) is negative definite on the subspace $V_\alpha$ forms that
can be written $\chi^\alpha=\chi^\alpha_0\wedge\omega$ ( it then equals a negative multiple of the norm square
of $\chi_0$),
\item the annihilator $V_\alpha^o$ of $V_\alpha$ with respect to $[,]$ in (\ref{bilinear form 2}),
consists  precisely of forms satisfying $\chi^\alpha\wedge \omega_q=0$ , and $V_\alpha\cap V_\alpha^o=\{0\}$,
\item  the form  (\ref{bilinear form 2}) is positive definite on $V^o_\alpha$,
\item  any $(n-q,1)$-form $\chi^\alpha$ can be decomposed uniquely
\begin{align*}
\chi^\alpha=\chi^\alpha_1+\chi^\alpha_0\wedge\omega
\end{align*}
with $\chi^\alpha_1\in V_\alpha^o$.
\end{itemize}

Since the point $x\in X$ is arbitrarily chosen, it follows that any $E\otimes L$-valued $(n-q,1)$-form
 $\chi$  can be decomposed uniquely
\begin{align*}
\chi=\chi_1+\chi_0\wedge\omega
\end{align*}
with $\chi_1\in V^o$, where $V^o$ is the annihilator of $V$ with respect to $[,]$ in (\ref{bilinear form 1})
and $V$ is the subspace of $E\otimes L$-valued $(n-q,1)$ forms on which the bilinear form $[,]$
in (\ref{bilinear form 1}) is negative definite. 

Now let $\chi=\overline{\partial}\gamma$. Since
$u=c_{n-q}\gamma\wedge\omega_q$ is $\overline{\partial}$-closed, we have that
\begin{align}
\overline{\partial}\gamma\wedge\omega_q=(-1)^{n-q-1}\gamma\wedge\overline{\partial}\omega_q.\label{dbar alpha}
\end{align}

Decomposing $\overline{\partial}\gamma=\chi_1+\chi_0\wedge\omega$ and plugging into (\ref{dbar alpha}), we have
\begin{align*}
\chi_0\wedge\omega\wedge\omega_q=(-1)^{n-q-1}\gamma\wedge\overline{\partial}\omega_q
.
\end{align*}

Since $\chi_0$ is of bidegree $(n-q-1,0)$, this means that
$|\chi_0|=O(|\gamma||\overline{\partial}\omega_q|)=O(|u||\overline{\partial}\omega_q|)$.
This means that the only possible  negative contribution of $[\overline{\partial}\gamma,\overline{\partial}\gamma]$
can be estimated by $c|u|^2$. If we also estimate the earlier error term
\begin{align*}
O(|\overline{\partial}^*u||\overline{\partial}\omega_{q-1}||u|)
\leq C_\varepsilon|u|^2+\varepsilon |\overline{\partial}^*u|^2,
\end{align*}
and collect all the terms and note that $\overline{\partial}\overline{\partial}^*u=\Box u$
if $u$ is $\overline{\partial}$-closed,  we get that
\begin{align*}
i\partial\overline{\partial}T_u \wedge\omega_{q-1}\geq(-2Re\langle\Box u,u\rangle+
\langle\Theta_{E\otimes L}\wedge\Lambda u,u\rangle+
(1-\varepsilon)|\overline{\partial}^*u|^2-C_\varepsilon|u|^2)\omega_n.
\end{align*}

Thus we get the following
\begin{prop}\label{ddbar formula}Let $(X,\omega)$ be a compact Hermitian manifold with
Hermitian metric $\omega$, $E$ and $L$ be holomorphic vector bundle of rank $r$ and holomorphic
line bundle respectively. Let $u$ be an $E\otimes L$-valued $(n,q)$-form. If $u$ is
$\overline{\partial}$-close, the following inequality holds
\begin{align*}
i\partial\overline{\partial}T_u\wedge\omega_{q-1}\geq(-2Re\langle\Box u,u\rangle+
\langle\Theta_{E\otimes L}\wedge\Lambda u,u\rangle-c|u|^2)\omega_n.
\end{align*}
The constant $c$ is equal to zero if $\overline{\partial}\omega_{q-1}=\overline{\partial}\omega_q=0$,
hence in particular if $\omega$ is K\"{a}hler.
\end{prop}
\begin{rmk}\label{bochner}
Actually, when $d\omega=0$, we can get the following identity for smooth $E\otimes L$-valued $(n,q)$-form $u$.
\begin{align*}
i\partial\overline{\partial}T_u\wedge\omega_{q-1}=(-2Re\langle\overline{\partial}
\overline{\partial }^* u,u\rangle+\langle\Theta_{E\otimes L}\wedge\Lambda u,u\rangle+
|\overline{\partial}\gamma_u|^2+|\overline{\partial}^*u|^2-|\overline{\partial}u|^2)\omega_n.
\end{align*}
\end{rmk}
\begin{proof}To get the equality, we only need to analyse the term
\begin{align*}
(-1)^{n-q+1}h_{\alpha\overline{\beta}}\overline{\partial}\gamma^\alpha\wedge\overline
{\overline{\partial}\gamma^\beta}e^{-\psi}.
\end{align*}
We need the following Lemma which is a variant of Lemma 4.2 in \cite{Ber10}.
\begin{lemma}\label{polarization lemma}Let $\xi$ be a $E\otimes L$ valued $(n-q,1)$ form. Then
\begin{align*}
ic_{n-q}(-1)^{n-q-1}h_{\alpha\overline{\beta}}\xi_\alpha\wedge\overline{\xi_\beta}
\wedge\omega_{q-1}=(|\xi|^2-|\xi\wedge\omega_q|^2)\omega_n.
\end{align*}
\end{lemma}
\begin{proof}First we observe that  the identity is pointwise. Fix arbitrary point $x_0\in X$,  we can
choose normal coordinates of $X$ centered at $x_0$ and choose normal trivialization of $E$ and any
trivialization of $L$, such that $\omega(x_0)=\sum dz_i\wedge d\overline{z}_i$ and
$h_{\alpha\overline{\beta}}(x_0)=\delta_{\alpha\beta}$. 

Then the question in hand is reduced to the case
considered in Lemma 4.2 in \cite{Ber10}. Thus we complete the proof of the Lemma.
\end{proof}

From Lemma \ref{polarization lemma}, Remark \ref{bochner} follows easily along the line of the proof of
Proposition \ref{ddbar formula}.
\end{proof}
It is worth to mention that from Remark \ref{bochner}, we can get the following estimate
\begin{prop}\label{bochner inequality}
Assume $(X,\omega)$ is a compact K\"{a}hler $n$-fold, $E$ and $L$ are holomorphic Hermitian vector bundles
and line bundles over $X$, and the curvature form $\Theta_{E\otimes L}$ is strictly Nakano positive, i.e.
$\Theta_{E\otimes L}\geq c\omega\otimes \mathds{1}_{E\otimes L}$ for some positive constant $c$.
Let $u$ be an $E\otimes L$ valued $(n,q)$-form. Then we have
\begin{align*}
qc\int_X|u|^2\omega_n+\int_X|\overline{\partial}\gamma_u|^2\omega_n
\leq \int_X|\overline{\partial}u|^2\omega_n+\int_X|\overline{\partial}^*u|^2\omega_n.
\end{align*}
\end{prop}

\begin{rmk}Remark \ref{bochner} and Proposition \ref{bochner inequality} can be used to prove vanishing
theorems and extension theorems for vector bundles, cf. \cite{Rau12, Rau14}.
\end{rmk}

\section{$\partial\overline{\partial}$-formula for the singular line bundle case}\label{Singular ddbar} We are
now concerned with the situation that the metric $h_L$ of the line
bundle $L$ is singular. Suppose that  the curvature
$\Theta_{h_L}\geq \gamma$ in the sense of current for some
continuous $(1,1)$-form, i.e. the line bundle is  quasi
pseudo-effective. In this case, the local potentials $\psi$ are
quasi-psh functions. 

From (\ref{dbardbarstar}), one can see that if
$\overline{\partial}\omega_{q-1}\neq 0$, one can not get an estimate
of the term
$O(|\overline{\partial}^*u^\alpha||\overline{\partial}\omega_{q-1}|)$,
since in this case our $\psi$ is singular. For this reason, in this
subsection, we work on compact K\"{a}hler manifold, i.e. the
Hermitian metric $\omega$ satisfies $d\omega=0$ on $X$.

 Let $u$ be an $E\otimes L$ valued $(n,q)$-form, the associated $(n-q,n-q)$-form $T_u$  defined by (\ref{Tu}).
 We have the following local data
\begin{itemize}
\item Smooth metric $h_{\alpha\overline{\beta}}$ of $E$ and singular metric $h_L=e^{-\psi}$ locally with
$\psi$ a quasi-psh ($L^1_{loc}$) function.
\item  Singular Chern connection $D=D'+\overline{\partial}$, with the $(1,0)$-connection matrix
$\Gamma_E\otimes \mathds{1}_L-\partial \psi\otimes \mathds{1}_E$, where $\Gamma_E$ is the connection
matrix of $E$ and $\partial\psi$ is a $(1,0)$-form with $L^1_{loc}$ coefficients by Lemma \ref{regularity lemma}.
\item  $\Theta_{E\otimes L}=\Theta_E\otimes \mathds{1}_L+\partial\overline{\partial}\psi\otimes
\mathds{1}_{E\otimes L}$, where $\partial\overline{\partial}\psi$ is a closed positive current.
\item  The commutator formula (\ref{curvature}) also holds: $\overline{\partial}D'+D'\overline{\partial}=
\Theta_{E\otimes L}$.
\item  The $\overline{\partial}^*$-operator is also a first order differential operator with $L^1_{loc}$-coefficents.
\item  The operator $\overline{\partial}\overline{\partial}^*$ also makes sense, hence the Laplace operator
$\Box$ makes sense as well. In fact,  we have already  get the explicit formula for
$\overline{\partial}\overline{\partial}^*$ in  (\ref{dbardbarstar}), from which one can get the conclusion.
\end{itemize}
By the same computation in Section \ref{ddbar smooth}, under the extra assumption that $d\omega=0$,
we can get the following
\begin{prop}\label{singular case}Let $(X,\omega)$ be a compact K\"{a}hler manifold, $(E,h^E)\rightarrow X$
be a holomorphic Hermitian vector bundle over $X$, and $(L,h_L)$ be a holomrphic pseudo-effective line
bundle with singular metric $h_L$ such that $i\Theta_{L,h_L}\geq \gamma$ with $\gamma$ a continuous real
$(1,1)$-form on $X$. Suppose that $u$ is an $E\otimes L$-valued $(n,q)$-form. Then we have the following equality
\begin{align*}
i\partial\overline{\partial}T_u\wedge\omega_{q-1}=(-2Re\langle\overline{\partial}\overline{\partial }^* u,u\rangle+
\langle\Theta_{E\otimes L}\wedge\Lambda u,u\rangle+|\overline{\partial}\gamma_u|^2+
|\overline{\partial}^*u|^2-|\overline{\partial}u|^2)\omega_n.
\end{align*}
Moreover, if $u$ is $\overline{\partial}$-closed, then we have
\begin{align*}
i\partial\overline{\partial}T_u\wedge\omega_{q-1}\geq(-2Re\langle\Box u,u\rangle+
\langle\Theta_{E\otimes L}\wedge\Lambda u,u\rangle)\omega_n.
\end{align*}
\end{prop}

\section{Proof of Theorem \ref{vector bundle}}\label{proof of vector bundle}

Have the vector bundle version $\partial\overline{\partial}$-formula in hand, the proof of Theorem
\ref{vector bundle} can be copied word by word from \cite{Ber02}.

Firstly, from the vector bundle version  $\partial\overline{\partial}$-formula, we can obtain that
\begin{align*}
i\partial\overline{\partial}T_u\wedge\omega_{q-1}\geq(-2Re\langle\Box u,u\rangle-c'|u|^2)\omega_n,
\end{align*}
Following the same argument (a standard calculation) in \cite{Ber02}, one can conclude that if
$u\in \mathcal{H}^{n,q}_{\leq \lambda}(X,L^k\otimes E)$ with $\overline{\partial}u=0$, and the metric
on $L$ as semi-positive curvature, then for $r<\lambda^{-1/2}$ and $r<c_0$,
\begin{align*}
\int_{|z|<r}|u|^2\omega_n\leq Cr^{2q}(\lambda+1)^q
\int_X|u|^2\omega_n,
\end{align*}
where the constants $c_0$ and $C$ are independent of $k$, $\lambda$ and the point $x$.

Secondly, one can get a pointwise norm estimate of $u$. In fact, fix
an arbitrary point $x\in X$. Choose a local coordinates $z$, near
$x$ such that $z(x)=0$ and
$\omega=\frac{i}{2}\partial\overline{\partial}|z|^2$ at $x$. Choose
local trivializations of $E$ and $L$ such that the metrics of $E$
and $L$ take the following form
\begin{align}\label{local metric}
h_{\alpha\overline{\beta}}(z)&=\delta_{\alpha\overline{\beta}}+O(|z|^2);\\
\psi(z)&=\sum\mu_j|z_j|^2+o(|z|^2).\notag
\end{align}

We do the following rescaling trick. For any form $u$, we express $u$ in terms of the trivializations
and local  coordinates and put
\begin{align*}
u^{(k)}(z)=u(z/\sqrt{k}),
\end{align*}
so that $u^{(k)}$ is defined for $|z|<1$ if $k$ is large enough. 

In the same time, the Laplacian is also
scaled in the following form
\begin{align*}
k\Box^{(k)}u^{(k)}=(\Box u)^{(k)}.
\end{align*}

As stated in \cite{Ber02}, it is not hard to see that if $\Box$ is defined by the metric $k\psi$ on $L^k$
and $h_{\alpha\overline{\beta}}$ on $E$, then $\Box^{(k)}$ is associated to the  line bundle metric
$(k\psi)(z/\sqrt{k})$ and the vector bundle metric $h_{\alpha\overline{\beta}}(z/\sqrt{k})$.

From (\ref{local metric}), we can see that  $\Box^{(k)}$ is associated to
\begin{align*}
\sum\mu_j|z_j|^2\otimes \delta_{\alpha\overline{\beta}}+o(1),
\end{align*}
hence converges to a $k$-independent elliptic operator. Hereafter, by almost the same argument as
in \cite{Ber02}, one can complete the proof of Theorem \ref{vector bundle}.

\section{Proof of Theorem \ref{main result} }\label{proof}
Now let $\mu:\widetilde{X}\rightarrow X$, $c$, $\lambda_j$ and $D_j$ be as in (II) of Section \ref{CAS}.
For any Hermitian vector bundle $(F,h_F)$ on $X$, we set $(\widetilde{F}, h_{\widetilde{F}}):=(\mu^*F,\mu^*h_F)$.

From Remark \ref{local decomposition}, we have that
  \begin{align*}
\widetilde{\varphi}:=\mu^*\varphi=\varphi\circ\mu=c\log|s_{D}|+\widetilde{\psi},
\end{align*}
where $s_D$ is the canonical section of
$\mathcal{O}_{\widetilde{X}}(-D)$ with $D=\sum_j\lambda_jD_j$, and
$\widetilde{\psi}$ is a smooth potential.

From (I) in Section \ref{CAS}, we have
\begin{align}
\mathcal{I}(h_{\widetilde{L}^p})=\mathcal{O}_{\widetilde{X}}(-\sum_j\lfloor c\lambda_jp\rfloor D_j).\notag
\end{align}

The main idea is to take advantage of the fact that $\mathcal{I}(h_{\widetilde{L}^p})$ is invertible and
we write $\widetilde{L}^p\otimes\mathcal{I}(h_{\widetilde{L}^p})$ as a tensor power of a fixed line bundle.

Fix $r\in \mathbb{N}$, $m\in \mathbb{N}^*$ such that $c=r/m$. Set $\widetilde{D}=rD=mc\sum_j\lambda_jD_j$,
$\widehat{L}=\widetilde{L}^m\otimes \mathcal{O}_{\widetilde{X}}(-\widetilde{D})$.

 Let $h_{\mathcal{O}_{\widetilde{X}}(-\widetilde{D})}$ be the singular Hermitian metric
 on $\mathcal{O}_{\widetilde{X}}(-\widetilde{D})$ given by $|s_D|^{2r}$ locally. Then the local potential
 of $h_{\mathcal{O}_{\widetilde{X}}(-\widetilde{D})}$ is $-r\log|s_D|$.

 Let  $h_{\widehat{L}}=h_{\widetilde{L}^m}\otimes h_{\mathcal{O}_{\widetilde{X}}(-\widetilde{D})}$ be
 the metric on $\widehat{L}$ induced by $h_{\widetilde{L}^m}$ and $h_{\mathcal{O}_{\widetilde{X}}(-\widetilde{D})}$.

It is easy to see that  the metric $h_{\widehat{L}}$ is smooth on $\widetilde{X}$, and
$i\Theta_{\widehat{L}, h_{\widehat{L}}}\geq 0$. In fact, the local weight $\widehat{\varphi}$ of $h_{\widehat{L}}$ is
\begin{align*}
\widehat{\varphi}&=m\widetilde{\varphi}-r\log|s_D|\\
&=m(\widetilde{\varphi}-c\log|s_D|)\\
&=m\widetilde{\psi}.
\end{align*}

By taking $\frac{i}{\pi}\partial\bar{\partial}$, we get that
 \begin{align*}
\frac{i}{\pi}\partial\bar{\partial}\widehat{\varphi}&=m\frac{i}{\pi}\partial\bar{\partial}\widetilde{\psi}\\
&=m\frac{i}{\pi}\partial\bar{\partial}(\widetilde{\varphi}-c\log|s_D|)\\
&=m(\frac{i}{\pi}\Theta_{\widetilde{L}, h_{\widetilde{L}}}-[D])\\
&=m\beta\geq 0,
\end{align*}
where the last equality follows from Remark \ref{local decomposition}. That is to say, $\widehat{L}$ can be
equipped with a smooth Hermitian metric with semi-positive curvature.

We observe that for $p'\in \mathbb{N}$,
\begin{align}
\mathcal{I}(h_{\widetilde{L}^{mp'}})=\mathcal{O}_{\widetilde{X}}(-p'\widetilde{D}), ~~\widehat{L}^{p'}=
\widetilde{L}^{mp'}\otimes \mathcal{I}(h_{\widetilde{L}^{mp'}})\notag.
\end{align}

Write $p=p'm+m'$ (where $c=r/m$ as above, $0\leq m'<m, p',m'\in \mathbb{N})$; then
$\lfloor c\lambda_jp\rfloor=r\lambda_jp'+\lfloor c\lambda_jm'\rfloor$. 

Now we want to prove that for
$p$ sufficiently large,
\begin{align}\label{dimension estimate}
\dim_\mathbb{C}H^{n,q}(\widetilde{X}, \widetilde{L}^p\otimes\widetilde{E}\otimes\mathcal{I}
(h_{\widetilde{L}^p}))\leq C(p')^{n-q}\leq Cp^{n-q},
\end{align}
The proof depends on the residue $m'=0,1,\cdots, m-1$.

\begin{itemize}
\item [(1)] $m'=0$, i.e.  $p=mp'$. Since the metric on $\widehat{L}$ is semipositive, then from
Theorem \ref{vector bundle}, we have that for $P'$ sufficiently large
\begin{align*}
\dim_\mathbb{C}H^{n,q}(\widetilde{X}, \widetilde{L}^p\otimes\widetilde{E}\otimes\mathcal{I}(h_{\widetilde{L}^p}))&=
\dim_\mathbb{C}H^{n,q}(\widetilde{X}, \widehat{L}^{p'}\otimes\widetilde{E})\\&\leq C_0(p')^{n-q}\leq C_0p^{n-q},
\end{align*}
where the constant $C_0$ is independent of  $p$.
\item [(2)] For $m'\neq 0$, we consider integers $p_{m'}=p'm+m'$, $p'\in \mathbb{Z}_+$. Set
\begin{align}
\widetilde{E}_{m'}=\widetilde{L}^{m'}\otimes\widetilde{E}
\otimes\mathcal{O}_{\widetilde{X}}(-\sum_j\lfloor c\lambda_jm'D_j\rfloor).\notag
\end{align}
Then we get that
\begin{align}
\widetilde{L}^{p_{m'}}\otimes \widetilde{E}\otimes \mathcal{O}_{\widetilde{X}}
(-\sum_j\lfloor c\lambda_jp_{m'}\rfloor D_j)=\widehat{L}^{p'}\otimes \widetilde{E}_{m'}.\notag
\end{align}
Then since $\widetilde{E}_{m'}$ is now a holomorphic line bundle, one can take a smooth Hermtian metric
on $\widetilde{E}_{m'}$, and by applying (\ref{dimension estimate}), we can get that for $p'$ sufficiently large and
\begin{align}
\dim_\mathbb{C}H^{n,q}(\widetilde{X}, \widetilde{L}^{p_{m'}}\otimes\widetilde{E}
\otimes\mathcal{I}(h_{\widetilde{L}^{p_{m'}}}))&=
\dim_\mathbb{C}H^{n,q}(\widetilde{X}, \widehat{L}^{p'}\otimes\widetilde{E}_{m'})\notag
\\&\leq C_{m'}(p')^{n-q}\leq C_{m'}p_{m'}^{n-q},\notag
\end{align}
where the constant $C_{m'}$ is independent of $p_m'$.
\item[(3)]To sum up, from the above $m$ cases, one can conclude that for  $p$ sufficiently large,
the following estimate holds:
 \begin{align}
\dim_\mathbb{C}H^{n,q}(\widetilde{X}, \widetilde{L}^p\otimes\widetilde{E}\otimes\mathcal{I}
(h_{\widetilde{L}^p}))\leq Cp^{n-q},\notag
\end{align}
where the constant  $C$ is independent of $p$.
\end{itemize}

Substituting $\widetilde{E}$ by $\widetilde{E}\otimes K^*_{\widetilde{X}}$, we get that
\begin{align}
\dim_\mathbb{C}H^{q}(\widetilde{X}, \widetilde{L}^p\otimes\widetilde{E}\otimes\mathcal{I}
(h_{\widetilde{L}^p}))\leq  Cp^{n-q}.\notag
\end{align}

We now apply Lemma \ref{canonical sheaf-modification} to each step of the blowing-up of $\mathscr{I}$
performed in the modification $\mu$ in the assumption. In doing so, we replace  $E$ by $E\otimes K^*_X$.

The hypothesis in Lemma \ref{canonical sheaf-modification} is satisfied, since our local weight $\varphi$
has analytic singularities and the centers of the blow-ups are included in the sigular locus of the metric.

Thus we can apply Lemma \ref{canonical sheaf-modification} finitely many times and we get that
for all $q\geq 0$, and $p$ large enough,
\begin{align}\label{reduction down}
H^{q}(X,L^p\otimes E\otimes \mathcal{I}(h_{L^p}))\simeq H^{q}(\widetilde{X}, \widetilde{L}^p\otimes \widetilde{E}
\otimes K_{\widetilde{X}}\otimes\widetilde{K^*_X}\otimes \mathcal{I}(h_{\widetilde{L}^p})).
\end{align}

Then applying the above  proof to the right-hand side of (\ref{reduction down}), we complete the proof of
Theorem \ref{main result}.

\section{A partial solution to Question \ref{Demailly question}}\label{application1}
Denote by  $V_k$  the support of the multiplier ideal  sheaf
$\mathscr{O}_X/\mathcal{I}(h_L^k)$. 

From the short exact sequence
\begin{align*}
0\rightarrow \mathscr{O}_X(E\otimes L^k)\otimes\mathcal{I}(h_L^k)\rightarrow &
\mathscr{O}_X(E\otimes L^k)\rightarrow \mathscr{O}_{V_k}(E\otimes L^k)\rightarrow 0,
\end{align*}
we can get a long exact sequence
\begin{align}\label{long exact sequence}
\cdots&\rightarrow H^q(X,\mathscr{O}_X(E\otimes L^k)\otimes\mathcal{I}(h_L^k))
\rightarrow H^q(X,\mathscr{O}_X(E\otimes L^k))\\&\rightarrow H^q(V_k,\mathscr{O}_{V_k}(E\otimes L^k))
\rightarrow H^{q+1}(X,\mathscr{O}_X(E\otimes L^k)\otimes\mathcal{I}(h_L^k))\rightarrow \cdots.\notag
\end{align}

Since $h_L$ is a singular metric with analytic singularities, from
Lemma \ref{skoda} (or by the strong Notherian property of the ideal
sheaf $\mathscr{O}_X/\mathcal{I}(h_L^k)$), we know that for
sufficiently large $k$, $V_k$ is stationary,
which is just the singular locus of the metric $h_L$. 

It follows from Theorem
\ref{main result} that
\begin{align}\label{application-estimate 1}
\dim H^q(X,\mathscr{O}_X(E\otimes L^k)\otimes \mathcal{I}(h_L^k))\leq Ck^{n-q}.
\end{align}

Suppose that the dimension of $V_k$ for $k$ large is $m$. We have that
\begin{align}\label{vanish of dimension}
\dim H^{q}(V_k,\mathscr{O}_{V_k}(E\otimes L))=0, ~~~\mbox{~~for~~~}~~q>m.
\end{align}

By combining (\ref{long exact sequence}), (\ref{application-estimate 1}) and (\ref{vanish of dimension}),
we obtain that
\begin{align*}
\dim H^q(X,\mathscr{O}_X(E\otimes L^k))\leq Ck^{n-q}, ~~~\mbox{~~~for~~~}~~~q>m.
\end{align*}

In conclusion, we get the following
\begin{thm}\label{application-partial answer}Let $X$ be a compact complex manifold, $E\rightarrow X$ be
a holomorphic vector bundle over $X$, and $L\rightarrow X$ be a holomorphic line bundle with a
singular Hermitian metric $h_L$ with algebraic singularities such that the curvature current
of $h_L$ is semi-positive. Assume that the dimension of the singular locus of $h_L$ is $m$.
Then for $q>m$, we have that
\begin{align*}
\dim H^q(X,\mathscr{O}_X(E\otimes L^k))\leq Ck^{n-q}.
\end{align*}
\end{thm}
\begin{rmk}
Theorem \ref{application-partial answer} is a  partial answer to Question \ref{Demailly question}.
\end{rmk}

The assumption of algebraic singularities can be weakened. In fact, we have the following 
\begin{thm}\label{analytic singularity}
	Let $X$ be a compact complex manifold, $E\rightarrow X$ be
	a holomorphic vector bundle over $X$, and $L\rightarrow X$ be a holomorphic line bundle with a
	singular Hermitian metric $h_L$ with  analytic singularities such that the curvature current
	of $h_L$ is semi-positive. Let $h$ be an arbitrarily smooth Hermitian metric of $L$, and set $e^{-\psi}=h_L/h$. Suppose that there is a small $\varepsilon>0$, such that $he^{-(1+\delta)\psi}$  are singular metrics of $L$ with semi-positive curvature current for $|\delta|<\varepsilon$. Assume that the dimension of the singular locus of $h_L$ is $m$.
	Then for $q>m$, we have that
	\begin{align*}
	\dim H^q(X,\mathscr{O}_X(E\otimes L^k))\leq Ck^{n-q}.
	\end{align*}
\end{thm}
 To prove Theorem \ref{analytic singularity}, we need the following Diophantine approximation theorem due to \'{E}mile Borel.
 
 \begin{lemma}
 		For every irrational number $c$, there are infinitely many frations $\frac{p}{q}$, such that 
 	\begin{align}\label{DA}
 	\Big|c-\frac{p}{q}\Big|<\frac{1}{\sqrt{5}q^2}.
 	\end{align}
 \end{lemma}

\begin{proof}[Proof of Theorem \ref{analytic singularity}]
	Since $h_L$ is with analytic singularities, we have that, locally,  $\psi$ can be written 
	\begin{align*}
	\psi\equiv \frac{c}{2}\log(\sum_{j=1}^N|f_j|^2) \mod \mathcal{C}^\infty.
	\end{align*} 
	For any $(p,q)$ satisfies (\ref{DA}), we have 
	\begin{align*}
	\Big|1-\frac{1}{c}\frac{p}{q}\Big|\leq \frac{1}{\sqrt{5}cq^2}
	\end{align*}
	Take sufficiently large $(p,q)$ such that $\frac{p}{q}$ satisfies (\ref{DA}) and $\frac{1}{\sqrt{5}cq^2}<\varepsilon$. Set $c_{p,q}=\frac{1}{c}\frac{p}{q}$, then $|c_{p,q}-1|<\varepsilon$. From the assumption, we have that $h_{p,q}:=he^{-c_{p,q}\psi}$ is a singular metric of $L$ with algebraic singularities such that the curvature current is semi-positive. In fact, 
	\begin{align*}
h_{p,q}	=&=he^{-(1+(c_{p,q}-1))\psi}=he^{-(1+\delta)\psi}, |\delta|<\varepsilon,\\
	c_{p,q}\psi&\equiv \frac{p/q}{2}\log(\sum_{j=1}^N|f_j|^2) \mod \mathcal{C}^\infty.
	\end{align*}
	But the singular locus of $h_{p,q}$ is exactly the same as the one of $h_L$. By applying Theorem \ref{application-partial answer}, we can complete the proof of Theorem \ref{analytic singularity}.

\end{proof}

\begin{rmk}
	From the proof, we can see that the $\varepsilon$ in the assumption of Theorem \ref{analytic singularity} can be chosen to be arbitrarily small.
\end{rmk}

\begin{rmk}If $(L,h_L)\rightarrow X$ is a singular metric with analytic singularities such that
the curvature current is semi-positive in the sense of current and
the singular locus of $h_L$ are isolated points, then we can see
that $L$ admits a smooth Hermitian metric with semi-positive
curvature. From this, we can conclude that if $L$ satisfies the
assumption in Theorem \ref{application-partial answer}, and
furthermore $L$ is not semi-positive, then the dimension of the
singular locus is positive.
\end{rmk}

\begin{rmk}\label{nef-not zero}A pseudo-effective line bundle $L$ is nef if there is a singular metric
on $L$ with semi-positive curvature current such that the Lelong number of the local potential is
zero everywhere. More precisely, the necessary and sufficient condition for a pseudo-effective
line bundle to be nef is characterized in \cite{Pau98a, Pau98b} by P\u{a}un.

We want to mention that one may not hope that for every nef line bundle $L$, there exists a
singular metric $h$ on $L$ with semi-positive curvature current, such that the Lelong number
of the local potential of $h$ is everywhere zero. Actually, it is closely related to the
so called non-K\"{a}hler locus or non-nef locus  which was  systematically  studied  in \cite{Bou04} and \cite{CT15}.
\end{rmk}

To finish this section, we mention an example of Demailly-Peternelle-Schneider in \cite{DPS} as a
supplement of Remark \ref{nef-not zero}. 

Let $\Gamma=\mathbb{C}/(\mathbb{Z}+\mathbb{Z}\tau)$, Im$\tau>0$,
be an elliptic curve and let $E$ be the rank $2$ vector bundle over $\Gamma$ defined by
\begin{align*}
E=\mathbb{C}\times \mathbb{C}^2/(\mathbb{Z}+\mathbb{Z}\tau)
\end{align*}
where the action of $\mathbb{Z}+\mathbb{Z}\tau$ is given  by the two automorphisms
\begin{align*}
g_1(x,z_1,z_2)&=(x+1,z_1,z_2);\\
g_{\tau}(x,z_1,z_2)&=(x+\tau,z_1+z_2, z_2),
\end{align*}
where the projection $E\rightarrow \Gamma$ is induced by the first projection $(x,z_1,z_2)\mapsto x$.
Then $\mathbb{C}\times \mathbb{C}\times \{0\}/(\mathbb{Z}+\mathbb{Z}\tau)$ is a trivial line
subbundle $\mathscr{O}\hookrightarrow E$, and the quotient
$E/\mathscr{O}\simeq \Gamma\times\{0\}\times \mathbb{C}$ is also trivial. 

Let $L$ be the line bundle
$L=\mathscr{O}_E(1)$ over the ruled surface $X=\mathbb{P}(E)$. From the exact sequence
\begin{align*}
0\rightarrow \mathscr{O}\rightarrow E\rightarrow \mathscr{O}\rightarrow 0,
\end{align*}
it is shown in \cite{DPS} that $L$ is nef over $X$. 

Moreover, the only possible metric $h$ of $L$
with semi-positive curvature is shown to be a singular metric with analytic singularities and
moreover $\frac{i}{\pi}\Theta_{L,h}=[C]$, where $[C]$ is the current of integration over a curve $C$.
For detailed computations, the reader is referred to see \cite[Example 1.7]{DPS}.

\section{Two vanishing theorems}\label{application3}

\begin{definition}[Kodaira-Iitaka dimension of a line bundle] For a holomorphic line bundle over
a compact complex manifold $X$, the Kodaira-Iitaka dimension of $L$ is defined to be
 \begin{align*}
 \kappa(L):=\limsup_{k\rightarrow +\infty}\frac{\log\dim_{\mathbb{C}} H^0(X,L^k)}{\log k}.
 \end{align*}
 \end{definition}

It is worth to  mention that for any compact complex manifold $X$ and a holomorphic line bundle
$L$ over $X$, if the Kodaira-Iitaka dimension of $L$ is non-negative, then there is a singular metric
$h_L$ with analytic singularities on $L$ such that the curvature current is semi-positive
in the sense of current. 

Moreover, by using sections of tensor powers $L^k$ of $L$, one can define
Siu's metric as follows:  for a basis $\{s_j^k\}_{j=1}^{N_k}$ of $H^0(X,L^k)$, we define a metric $\varphi_k$ by
\begin{align}\label{Siu metric}
\varphi_k:=\frac{1}{2k}\log\sum_{j=1}^{N_k}|s_j^k|^2.
\end{align}

Taking a convergent series $\{\varepsilon_k\}^\infty_{k=1}$, one can define a metric $h_{siu}$
on $L$ whose local weight is equal to $\log\sum_{k=1}^{\infty}\varepsilon_k e^{\varphi_k}$. This type of metric
is called Siu's metric which was first introduced by Siu and plays important role in \cite{Siu98}.

Siu's metric $h_{siu} $ and the associated multiplier ideal sheaf $\mathcal{I}(h_{siu})$ depend on the
choice of $\{\varepsilon_k\}_{k=1}^\infty$, but $h_{siu}$ always admits an analytic Zariski decomposition,
i.e. $H^0(X, L^k)=H^0(X,L^k\otimes \mathcal{I}(h^k_{siu}))$.

\begin{thm}\label{vanishing}Let $X$ be a compact K\"{a}hler manifold and $L$ be a holomorphic line bundle
over $X$. Suppose that $L$ is pseudo-effective, and the singular metric $h_{\min}$  with minimal
singularities of $L$ is with algebraic singularities. Then we have that
\begin{align*}
H^q(X,\mathscr{O}_X(K_X\otimes L)\otimes\mathcal{I}(h_{\min}))=0  ~~~~~\mbox{~~~for~~~~}~~~q>n-\kappa(L).
\end{align*}
\end{thm}

To prove the above Theorem, we need the following Theorem which is a consequence of injectivity theorem.
\begin{thm}[{\cite[Corollary 3.3]{Mat14} }]\label{Cor 3.3}
Let $(L,h_L)$ and  $(M,h_M)$ be line bundles with singular metrics on a compact K\"{a}hler manifold $X$.
Assume the following conditions:
\begin{itemize}
\item There exists a subvariety $Z$ on $X$ such that $h_L$ and $h_M$ are smooth on $X\setminus Z$.
\item $\sqrt{-1}\Theta_{h_L}(L)\geq \gamma$ and $\sqrt{-1}\Theta_{h_M}(M)\geq \gamma$ for some smooth
$(1,1)$-form $\gamma$ on $X$.
\item $\sqrt{-1}\Theta_{h_L}(L)\geq 0$ on $X\setminus Z$.
\item $\sqrt{-1}\Theta_{h_L}(L)\geq \varepsilon \sqrt{-1}\Theta_{h_M}(M)$ on $X\setminus Z$ for some
positive number $\varepsilon>0$.
\end{itemize}
Assume that $h^q(X,\mathscr{O}_X(K_X\otimes L)\otimes \mathcal{I}(h_L))$ is not zero. Then we have
\begin{align*}
\dim H^0_{bdd,h_M}(X,M)\leq h^q(X,\mathscr{O}_X(K_X\otimes L\otimes M)\otimes \mathcal{I}(h_Lh_M)),
\end{align*}
where $H^0_{bdd,h_M}(X,M)$ is the space of sections of $M$ with bounded norm
\begin{align*}
H^0_{bdd,h_M}(X,M):=\{s\in H^0(X,M)|\sup_X|s|_{h_M}<+\infty\}.
\end{align*}
\end{thm}

\begin{proof}[Proof of Theorem \ref{vanishing}]  Suppose to the contrary, we assume that
$h^q(X,\mathscr{O}_X(K_X\otimes L)\otimes \mathcal{I}(h_{\min}))$  for $q>n-k(L)$ is not zero.

Since $h_{\min}$ is of minimal singularities, it admits an analytic Zariski decomposition, which means that
\begin{align*}
h^0(X,L^{k-1})=h^0_{bdd,h^{k-1}_{\min}}(X,\mathscr{O}_X(L^{k-1}))
\leq h^q(X,\mathscr{O}_X(K_X\otimes L^k)\otimes\mathcal{I}(h^k_{\min})),
\end{align*}
where the equality follows from the property that $h_{\min}$ is a singular metric with minimal singularities
and the   inequality follows from Theorem \ref{Cor 3.3}. 

By the definition of Kodaira-Iitaka dimension $\kappa(L)$,
we have that
\begin{align*}
\limsup_{k\rightarrow +\infty}\frac{h^0(X,L^{k-1})}{(k-1)^{\kappa(L)}}>0.
\end{align*}

On the other hand, by Theorem \ref{main result}, we have
$h^q(X, \mathscr{O}_X(K_X\otimes L^k)\otimes\mathcal{I}(h^k_{\min}))=O(k^{n-q})$ as letting $k$ go to infity.
It is a contradiction to the inequality $>n-\kappa(L)$.
\end{proof}

\begin{rmk}
The injectivity theorem used in the proof of Theorem \ref{Cor 3.3} has been already proved for arbitrary singular metrics in \cite{Mat16}.
\end{rmk}

\begin{rmk}Theorem \ref{vanishing} is a generalization of Theorem 1.4 (1) in \cite{Mat14} from the case of
smooth projective manifold to compact K\"{a}hler manifold under the
assumption that the singular metric with minimal singularities on
$L$ is with algebraic singularities.
\end{rmk}

\begin{rmk}Metrics with minimal singularities do not always have algebraic singularities (see \cite{Mat14}
and reference therein).
\end{rmk}

 By the same argument as in the proof of Theorem \ref{vanishing}, we can obtain the following
\begin{thm}
Let $X$ be a compact K\"{a}hler manifold and $L$ be a holomorphic
line bundle with non-negative Kodaira-Iitaka dimension over $X$.
Suppose that $L$ is pseudo-effective and the Siu's metric $h_{siu}$
of $L$ is with algebraic singularities. Then we have that
\begin{align*}
H^q(X,\mathscr{O}_X(K_X\otimes L)\otimes\mathcal{I}(h_{siu}))=0  ~~~~~\mbox{~~~for~~~~}~~~q>n-\kappa(L).
\end{align*}
\end{thm}

\end{document}